\documentclass[11pt]{article}
\usepackage[margin=1.1in]{geometry}
\usepackage{amsmath,amssymb,amsthm,mathtools}
\usepackage{enumitem}
\usepackage[colorlinks=true,linkcolor=black,citecolor=black,urlcolor=black,filecolor=black]{hyperref}

\newtheorem{theorem}{Theorem}[section]
\newtheorem{lemma}[theorem]{Lemma}
\newtheorem{proposition}[theorem]{Proposition}
\newtheorem{corollary}[theorem]{Corollary}
\theoremstyle{remark}
\newtheorem{remark}[theorem]{Remark}

\DeclareMathOperator{\rank}{rank}
\DeclareMathOperator{\dist}{dist}

\newcommand{\R}{\mathbb{R}}
\newcommand{\C}{\mathbb{C}}
\newcommand{\F}{\mathcal{F}}
\newcommand{\eps}{\varepsilon}
\newcommand{\ind}[1]{\mathbf{1}_{#1}}
\newcommand{\HS}{\mathrm{HS}}
\newcommand{\logp}{\log_{+}}
\newcommand{\lnp}{\ln_{+}}

\title{Scale-free Hankel factorization and explicit one-dimensional plunge
bounds for time--frequency localization operators}
\author{Ahmadreza Azimifard}
\date{}

\begin{document}
\maketitle
\renewcommand{\thefootnote}{\arabic{footnote}}

\begin{abstract}
The plunge region records the spectral transition of a time--frequency
localization operator. Let $A_0,B_0\subset\R$ be bounded measurable sets of
positive measure with finite boundaries, and let
$S_{cA_0,B_0}=P_{cA_0}Q_{B_0}P_{cA_0}$. We prove that, for every $c>0$ and
$0<\eps<1/2$,
\[
\Lambda_\eps \;\le\; C(A_0,B_0)\,\widetilde L\,
\Bigl(1+\lnp\!\frac{c\,a}{\widetilde L}\Bigr),
\qquad \widetilde L=\ln\frac{1}{\eps(1-\eps)},
\]
where $\lnp x=\max(0,\ln x)$. If $M$ and $K$ count the interval components and
$a$ and $b$ are their maximal lengths, one may take
$C(A_0,B_0)=63MK(1+b)$. The estimate is uniform in $c$ and $\eps$ and becomes
$O(\widetilde L)$ when $\widetilde L\ge ca$.

On the range $c\ge2$ and $\alpha^{-c}<\eps<1/2$ covered by the
parallelepiped theorem of Kulikov and Dam Larsen (2026), this recovers its
$d=1$ order; their complementary very-small-threshold result is sharper. We
give an independent direct proof with explicit geometry dependence and a
single formulation for all $c$ and $\eps$. It acts
on the off-diagonal factor $T=P_{(cA_0)^c}Q_{B_0}P_{cA_0}$: an exact
one-dimensional oscillation factorization reduces each far-field piece to two
modulated copies of the scale-free Hankel kernel $1/(2\pi(s+r))$,
Bernstein-ellipse approximation gives uniform
singular-value decay, and a Taylor-rank estimate controls the boundary layer.
Choosing this width at the spectral depth absorbs finer scales and leaves only
$\lnp(ca/\widetilde L)$ dyadic scales, yielding the self-improving logarithm.
It also identifies tangential phase dependence as an obstruction to this
factorization in nonproduct higher-dimensional geometry.
\end{abstract}

\section{Introduction and statement of results}

\subsection{Setup}
Throughout, the Fourier transform is normalized as
\[
\F f(\xi)=\widehat f(\xi)=\int_\R f(x)e^{-2\pi i x\xi}\,dx ,
\]
a unitary operator on $L^2(\R)$. For a measurable set $E\subset\R$ let
$P_E$ denote multiplication by $\ind{E}$ and
$Q_E=\F^{-1}P_E\F$ the corresponding frequency projection. For bounded
measurable $A_0,B_0\subset\R$ and $c>0$ put $A=cA_0$, $B=B_0$, and
\[
S \;=\; S_{cA_0,B_0}\;=\;P_{A}\,Q_{B}\,P_{A} .
\]
Then $S$ is self-adjoint, $0\le S\le I$, and compact (Lemma~\ref{lem:compact}
below), so its nonzero spectrum is a sequence
$\lambda_1\ge\lambda_2\ge\cdots>0$ of eigenvalues (with multiplicity)
contained in $(0,1]$. The \emph{plunge count} is
\[
\Lambda_\eps \;=\; \Lambda_\eps(cA_0,B_0)
\;=\;\#\{n:\ \eps<\lambda_n(S)<1-\eps\},\qquad 0<\eps<\tfrac12 .
\]
KDL count plunge eigenvalues with the half-open convention
$\#\{n:\ \eps<\lambda_n(S)\le1-\eps\}$; every bound below extends to that
count with the same constant, by applying it at $\eps'<\eps$ and letting
$\eps'\uparrow\eps$ (the bound is continuous in $\eps$).

Kulikov and Dam Larsen \cite{KDL} (henceforth KDL) proved that if one of
$A_0,B_0\subset\R^d$ is a finite disjoint union of parallelepipeds and the
other is bounded with finite upper Minkowski boundary content, then
$\Lambda_\eps\lesssim c^{d-1}LR$ with $L=\log(1/\eps)$,
$R=\log(\alpha c/L)$, on the range $c\ge2$ and
$\alpha^{-c}<\eps<1/2$. They also give a different, sharper description of
the lower-half count when $\eps\le\alpha^{-c}$. For general sets of finite
upper Minkowski boundary content they prove
$\Lambda_\eps\lesssim c^{d-1}LR^2$. They conjectured
(see the discussion following \cite[Thm.~1.6]{KDL}) that the sharp bound $c^{d-1}LR$ holds in
general. In $d=1$ the hypothesis ``finite upper $0$-dimensional Minkowski
boundary content'' is equivalent to $\partial A_0$, $\partial B_0$ being
finite sets (Remark~\ref{rem:mink}), so $A_0$ and $B_0$ are, up to null sets,
finite unions of intervals; a parallelepiped in $d=1$ is an interval, so
both sets satisfy the hypotheses of \cite[Thm.~1.3]{KDL} and the conjecture
in $d=1$ follows from that theorem. The $d=1$ statement is therefore not
open, and we do not claim it as new. What we give is an \emph{independent}
proof in $d=1$: it does not invoke KDL's parallelepiped theorem or their
prolate estimates for $S$, but it does share the overall off-diagonal and
Schatten-quasi-norm architecture that we credit explicitly below. There is
also a closer kernel-level antecedent. KDL's separated operator $J_r$ has an
oscillatory Cauchy kernel, and \cite[\S5.1]{KDL} obtains scale-uniform
singular-value decay by expanding it into rank-two pieces. The present route
reorganizes the one-dimensional analysis directly in boundary-distance
coordinates: it avoids a Whitney decomposition, keeps the exterior distance
variable on a full half-line, removes the oscillations by modulation, and
approximates the resulting scale-free Hankel kernel by Bernstein ellipses,
with a separate Taylor-rank estimate for the boundary layer.

What this route adds to the $c^{d-1}LR$ part of
\cite[Thm.~1.3]{KDL} in $d=1$ is an explicit constant, with its dependence
on the geometry displayed, and a single statement covering all $c>0$ and
all $0<\eps<\tfrac12$. It is not an improvement on KDL's sharper
very-small-threshold description; its advantages there are the uniform
formulation and explicit constants. On the $LR$ range, KDL assumes $c\ge2$
and $\alpha^{-c}<\eps<\tfrac12$ and states the bound with an unspecified
implied constant.

\subsection{Main results}

We write $\lnp x=\max(0,\ln x)$ and $\logp x = \max(0,\log_2 x)$.

\begin{theorem}\label{thm:main}
Let $A_0,B_0\subset\R$ be bounded measurable sets of positive measure with
$\partial A_0$ and $\partial B_0$ finite. Let $M$ and $K$ be the number of
connected components of $\operatorname{int}A_0$ and
$\operatorname{int}B_0$ respectively, let $a$ be the maximal component
length of $\operatorname{int}A_0$ and $b$ the maximal component length of
$\operatorname{int}B_0$. Then for all $c>0$ and all $0<\eps<1/2$,
\[
\Lambda_\eps(cA_0,B_0)\;\le\;
C_0\,MK\,(1+b)\,\widetilde L\,
\Bigl(1+\lnp\frac{c\,a}{\widetilde L}\Bigr),
\qquad
\widetilde L:=\ln\frac{1}{\eps(1-\eps)},
\]
where $C_0$ is an absolute constant; one may take $C_0=63$.
\end{theorem}

The quantities $M,K,a,b$ are read off a chosen representative of $A_0,B_0$
and are not invariant under modification by a null set, whereas
$S_{cA_0,B_0}$, and hence $\Lambda_\eps$, are. Deleting one interior point
of $A_0=(0,1)$ leaves the operator untouched but replaces $(M,a)=(1,1)$ by
$(2,\tfrac12)$. The theorem holds for every representative with finite
boundary, so its right-hand side may be read with the infimum taken over
all such representatives.

Note $\ln(1/\eps)\le \widetilde L\le \ln(1/\eps)+\ln 2$, so
$\widetilde L\asymp L:=\ln(1/\eps)$. Theorem~\ref{thm:main} requires no
lower bound on $c$ and no relation between $\eps$ and $c$. Specializing:

\begin{corollary}[The Kulikov--Dam Larsen bound in $d=1$]\label{cor:kdl}
With $A_0,B_0$ as in Theorem~\ref{thm:main}, for every $\alpha\ge4$,
$c\ge2$ and $\alpha^{-c}<\eps<1/2$,
\[
\Lambda_\eps(cA_0,B_0)\;\le\;
C(A_0,B_0)\,\log\frac1\eps\,
\log\!\Bigl(\frac{\alpha c}{\log(1/\eps)}\Bigr),
\]
with $C(A_0,B_0)= 4 C_0\,MK\,(1+b)(2+\lnp a)$, valid when $\log$ is read as
$\ln$ or as $\log_2$. Since $c^{d-1}=1$ in $d=1$, this is the conjectured
bound; as noted above it also follows from \cite[Thm.~1.3]{KDL}, and the
point of stating it here is that it comes out of the present proof with an
explicit constant.
\end{corollary}

\begin{remark}\label{rem:antecedents}
For a single pair of intervals and fixed $\eps$, Theorem~\ref{thm:main}
gives $\Lambda_\eps=O(\ln(1/\eps)\ln c)$, the classical Landau--Widom order
\cite{LandauWidom}; what the theorem adds is uniformity in $(\eps,c)$
jointly, down to $\eps$ exponentially small in $c$, with the self-improving
factor $\ln(ca/\widetilde L)$ in place of $\ln(ca)$.

The $d=1$ literature is substantial, and the comparison deserves to be
precise. Quantitative non-asymptotic antecedents include the two-sided
eigenvalue bounds of Israel \cite{Israel15}, the upper bounds of Osipov
\cite{Osipov}, the spectral-decay estimates of Bonami--Karoui
\cite{BonamiKaroui}, and, in the pre-plunge regime, the sharp estimates of
Kulikov \cite{Kulikov24,Kulikov26}. Two comparisons bear directly on what is
claimed here. Karnik--Romberg--Davenport \cite{KRD} give, for a \emph{single}
pair of bounded intervals, a completely explicit bound valid for every $c$
and all $\eps\in(0,\tfrac12)$, with a numerical constant far better than
ours; it lacks the self-improving logarithm, so it does not degenerate to
$O(\widetilde L)$ once $\widetilde L\ge ca$, and it says nothing about
unions. More directly, \cite[Thm.~1.3]{KDL} covers the whole hypothesis
class of Theorem~\ref{thm:main} --- not merely the interval case --- and
yields the same shape of bound. The higher-dimensional programme of Hughes,
Israel and Mayeli \cite{HIM} belongs to the same circle of ideas but is
aimed at $d\ge2$.

What this paper contributes is accordingly the route rather than the
inequality, together with the constants that route produces: they are
explicit, depend on the geometry only through $(M,K,a,b)$, and allow finite
unions on \emph{both} sides. The statement covers all $c>0$ and all
$0<\eps<\tfrac12$ with no relation imposed between them, and
Corollary~\ref{cor:kdl} holds for
\emph{every} $\alpha\ge4$ with a constant independent of $\alpha$, where
\cite[Thm.~1.3]{KDL} asserts the existence of one such $\alpha$.
\end{remark}

\subsection{Strategy}\label{sec:strategy}

Write $T=P_{A^c}Q_BP_A$, so that $S-S^2=T^*T$ and eigenvalues of $S$ in
$(\eps,1-\eps)$ produce singular values of $T$ above
$t=\sqrt{\eps(1-\eps)}$ (Lemma~\ref{lem:plunge}). All work happens on the
singular values $s_n(T)$, which we control in the Schatten quasi-norm
$\|T\|_p^p=\sum_n s_n(T)^p$ at the $\eps$-tuned exponent
\[
p=\frac{1}{\ln\bigl(1/(\eps(1-\eps))\bigr)}=\frac1{\widetilde L}\in(0,1],
\qquad\text{so that}\qquad
t^{-p}=e^{1/2},
\]
whence $\Lambda_\eps\le e^{1/2}\|T\|_p^p$ by quasi-norm Markov. This opening
is not ours. The reduction to $T=P_{A^c}Q_BP_A$, the passage to Schatten
$p$-quasi-norms at exactly this $\eps$-tuned exponent, the resulting factor
$e^{1/2}$, the use of subadditivity to reassemble the pieces, and the
observation that enlarging the outer set only increases singular values are
all in \cite[\S1.4]{KDL}, where the $p$-quasi-norm method is in turn
credited to Hughes, Israel and Mayeli \cite{HIM}. The present reorganization
begins with the exact one-dimensional reductions in move (1); the
scale-uniform separated-kernel estimate in move (2) has the KDL analogue
described above, but is obtained here on a different geometry by a
Hankel--Bernstein argument. The proof of
\[
\|T\|_p^p\;\lesssim\;p^{-1}\bigl(1+\lnp(ca/\widetilde L)\bigr)
\;=\;\widetilde L\bigl(1+\lnp(ca/\widetilde L)\bigr)
\]
has four moves.

\emph{(1) Exact reductions.} Frequency components of $B$ are split off by
Rotfel'd subadditivity and each is centred by a modulation (which commutes
with all spatial projections); spatial components of $A$ are split off
likewise, using $P_{A^c}=P_{A^c}P_{I^c}$ to dominate each piece by the
single-interval off-diagonal operator $P_{I^c}Q_{B^\circ}P_I$. A
left/right split and the substitutions $s=\dist(x,I)$, $r=|y-q|$, where
$q\in\partial I$ is the endpoint facing the chosen side, turn each one-sided
piece into the \emph{Hankel-type} operator
$\Gamma:L^2(0,\ell)\to L^2(0,\infty)$ with kernel
$\sin(\pi b(s+r))/(\pi(s+r))$. All steps preserve or dominate singular
values exactly (Section~\ref{sec:reductions}).

\emph{(2) Oscillation factorization (the $d=1$ mechanism).} Since the
kernel depends on $s,r$ only through $s+r$,
\[
\frac{\sin(\pi b(s+r))}{\pi(s+r)}
=\underbrace{e^{i\pi bs}}_{\text{unimodular in }s}
\underbrace{e^{i\pi br}}_{\text{unimodular in }r}\frac{1}{2\pi i(s+r)}
\;-\;
e^{-i\pi bs}\,e^{-i\pi br}\,\frac{1}{2\pi i(s+r)} .
\]
Multiplication by a unimodular function of the output (resp.\ input)
variable is unitary, so each of the two summands has \emph{exactly} the
singular values of the bandwidth-free Hankel kernel $1/(2\pi(s+r))$. The
far-field piece is their difference, and Ky Fan converts the statement about
the summands into a bound on the piece that is uniform in the bandwidth
(Section~\ref{sec:keyestimates}).

\emph{(3) Two quantitative estimates.} (a) A Bernstein-ellipse/Chebyshev
expansion in the $r$-variable alone gives, uniformly in the scale $h>0$,
\[
s_{N+1}\Bigl(\tfrac{1}{2\pi(s+r)}\ \text{on}\
(0,\infty)\times[h,2h]\Bigr)\;\le\;4^{-N},
\]
hence each dyadic far-field piece $A_k=\Gamma P_{(2^kD,2^{k+1}D]}$ obeys
$s_n(A_k)\le 2^{3-n}$ and $\|A_k\|_p^p\le 4.5/p$. (b) A Taylor-rank bound
on $Q_{B^\circ}P_{(0,D)}$ shows that at the width $D=1/p$ used below the
boundary layer satisfies $\|\Gamma P_{(0,D)}\|_p^p\le(\pi e\,b+7)/p$
(Section~\ref{sec:keyestimates}).

\emph{(4) Assembly with the boundary-layer width $D=1/p=\widetilde L$.}
Rotfel'd subadditivity over the boundary layer and the
$S=\lceil\logp(\ell/D)\rceil$ dyadic far-field scales gives
$\|\Gamma\|_p^p\le C_b(1+\logp(ca/\widetilde L))/p$; Markov finishes. The
choice $D=\widetilde L$ is what converts the naive factor $\ln(ca)$ into
$\ln(ca/\widetilde L)$: eigenvalue-counting depth $\widetilde L$ is traded
against geometric scales below width $\widetilde L$, which are absorbed
into a single rank-$O(\widetilde L)$ block
(Section~\ref{sec:assembly}).

A decomposition that is dyadic in \emph{one variable only}, with the other
variable kept global on a half-line, is essential: the per-scale estimate
(3a) is then uniform and exponential, and plain $p$-subadditivity across the
$O(\log)$ scales suffices. No almost-orthogonality is used. Exact
orthogonality of both initial and final spaces makes Rotfel'd subadditivity
an equality, so orthogonality alone cannot improve this generic
exponent-$p$ estimate (Remark~\ref{rem:cotlar}).

Section~\ref{sec:d1} states exactly which steps fail in $d>1$ and why.

\section{Preliminaries}\label{sec:prelim}

\subsection{Singular values}
For a compact operator $X$ between Hilbert spaces, $s_1(X)\ge s_2(X)\ge
\cdots$ denote its singular values (eigenvalues of $(X^*X)^{1/2}$, with
multiplicity), equivalently approximation numbers
$s_{n+1}(X)=\inf\{\|X-F\|:\rank F\le n\}$. Set
\[
n(t;X)=\#\{n\ge1:\ s_n(X)>t\},\qquad
\|X\|_p^p=\sum_{n\ge1}s_n(X)^p\quad(0<p\le1).
\]

\begin{lemma}\label{lem:svcalc}
Let $X,Y$ be compact, $U,V$ bounded. Then:
\begin{enumerate}[label=(\roman*),itemsep=1pt]
\item $s_n(UXV)\le\|U\|\,\|V\|\,s_n(X)$ for all $n$.
\item (Fan \cite{Fan,GK}) $s_{m+n+1}(X+Y)\le s_{m+1}(X)+s_{n+1}(Y)$.
\item $s_{N+1}(X)\le\|X-F\|\le\|X-F\|_{\HS}$ for every $F$ with
$\rank F\le N$.
\item (Rotfel'd \cite{Rotfeld}; see also \cite{SimonTI})
$\|X+Y\|_p^p\le\|X\|_p^p+\|Y\|_p^p$ for $0<p\le1$.
\item (Markov) $n(t;X)\le t^{-p}\|X\|_p^p$ for every $t>0$, $0<p\le1$.
\item If $U,V$ are multiplications by unimodular functions on the output
and input spaces respectively, then $s_n(UXV)=s_n(X)$ for all $n$.
\end{enumerate}
\end{lemma}

\begin{proof}
(i) and (iii) are immediate from the approximation-number characterization
($\|U(X-F)V\|\le\|U\|\|V\|\|X-F\|$ and $\rank(UFV)\le\rank F$; for (iii),
$\|\cdot\|\le\|\cdot\|_{\HS}$). (ii) is Ky Fan's inequality: if
$\rank F\le m$, $\rank G\le n$ then $\rank(F+G)\le m+n$ and
$\|X+Y-(F+G)\|\le\|X-F\|+\|Y-G\|$; take infima. (iv) is the Rotfel'd
inequality. (v): each of the $n(t;X)$ terms with $s_n>t$ contributes more
than $t^p$ to $\|X\|_p^p$. (vi): $U,V$ are unitary.
\end{proof}

We will also use repeatedly that if $\Pi$ is an orthogonal projection then
$s_n(X\Pi)\le s_n(X)$ and $s_n(\Pi X)\le s_n(X)$, a special case of (i).

\subsection{Structure of the sets and compactness}

\begin{lemma}\label{lem:structure}
Let $E\subset\R$ be bounded measurable with $\partial E$ finite and
$|E|>0$. Then there exist $1\le m\le\#\partial E$ disjoint bounded open
intervals $J_1,\dots,J_m$ such that $E\,\triangle\,(J_1\cup\dots\cup J_m)$
is a finite set, hence Lebesgue-null. Consequently
$P_E=P_{J_1}+\dots+P_{J_m}$ and $Q_E=Q_{J_1}+\dots+Q_{J_m}$ as operators
on $L^2(\R)$.
\end{lemma}

\begin{proof}
$U:=\operatorname{int}E$ is open, hence a countable disjoint union of open
intervals $J_k$ (its connected components), each bounded since $E$ is. Let
$J=(\alpha,\beta)$ be a component. Then $\alpha\in\overline{U}\subset
\overline E$ and $\alpha\notin\operatorname{int}E$ (otherwise a
neighbourhood of $\alpha$ lies in $U$ and $J$ would not be a maximal
component), so $\alpha\in\partial E$; similarly $\beta\in\partial E$. A
point of $\R$ is the left endpoint of at most one component and the right
endpoint of at most one component, so counting endpoint incidences,
$2m\le2\,\#\partial E$, i.e.\ the number $m$ of components is at most
$\#\partial E$; in particular it is finite. Finally
$E\subset\operatorname{int}E\cup\partial E$ and
$\operatorname{int}E\subset E\cup\partial E$ give
$E\,\triangle\,U\subset\partial E$, a finite set. $|E|>0$ forces $m\ge1$.
Multiplication operators are insensitive to null modifications, and
$\ind{U}=\sum_k\ind{J_k}$ pointwise gives $P_E=\sum_kP_{J_k}$; conjugating
by $\F$ gives the statement for $Q$.
\end{proof}

\begin{remark}\label{rem:mink}
For compact $F\subset\R$, the upper $0$-dimensional Minkowski content
$\mathcal M^{*0}(F)=\limsup_{\delta\downarrow0}|F_\delta|/(2\delta)$
(where $F_\delta$ is the $\delta$-neighbourhood) is finite iff $F$ is
finite. Indeed, if $F=\{x_1,\dots,x_m\}$ then $|F_\delta|\le2\delta m$. If
$F$ is infinite, choose any $m$ distinct points of $F$; for $\delta$ less
than half their minimal separation, $|F_\delta|\ge2\delta m$, so
$\mathcal M^{*0}(F)\ge m$ for every $m$. Thus the hypothesis of
\cite{KDL} (finite upper Minkowski boundary content) coincides in $d=1$
with the hypothesis of Theorem~\ref{thm:main}.
\end{remark}

\begin{lemma}\label{lem:compact}
$S=P_AQ_BP_A$ is self-adjoint, $0\le S\le I$, and Hilbert--Schmidt; in
particular compact. The off-diagonal operator $T=P_{A^c}Q_BP_A$ is
Hilbert--Schmidt as well, with
$\|T\|_{\HS}^2\le|A|\,|B|$.
\end{lemma}

\begin{proof}
$Q_B=\F^{-1}P_B\F$ is an orthogonal projection, so
$S=(Q_B^{1/2}P_A)^*(Q_B^{1/2}P_A)$ with $Q_B^{1/2}=Q_B$ shows $S\ge0$, and
$S\le P_A\le I$ in the form sense since $Q_B\le I$. Write
$g=\F^{-1}\ind{B}$; since $B$ is bounded, $\ind B\in L^2$, so
$g\in L^2(\R)$ and, $\F$ being unitary,
$\|g\|_{L^2}^2=\|\ind B\|_{L^2}^2=|B|$. Then $Q_B$ has convolution kernel
$g(x-y)$, and $S$ has kernel $\ind A(x)g(x-y)\ind A(y)$, with
\[
\iint_{A\times A}|g(x-y)|^2\,dx\,dy
\;\le\;\int_A\!\!\int_\R|g(x-y)|^2\,dx\,dy
\;=\;|A|\,\|g\|_{L^2}^2=|A||B|<\infty .
\]
The same computation applies verbatim to $T$, whose kernel is
$\ind{A^c}(x)g(x-y)\ind A(y)$: enlarging $A^c$ to $\R$ in the $x$-variable,
\[
\|T\|_{\HS}^2=\int_A\!\!\int_{A^c}|g(x-y)|^2\,dx\,dy
\;\le\;\int_A\!\!\int_\R|g(x-y)|^2\,dx\,dy=|A||B|<\infty .
\]
Finiteness here comes from $|A|<\infty$ together with $g\in L^2$; the
unboundedness of $A^c$ costs nothing, because the $x$-integration is
performed against the full line for each fixed $y\in A$. It is the
boundedness of the \emph{inner} set $A$ that is used, and this is the only
place it is needed.
\end{proof}

\subsection{From the plunge region to the off-diagonal operator}

\begin{lemma}\label{lem:plunge}
Let $T=P_{A^c}Q_BP_A$. Then $S-S^2=T^*T$ and, for $0<\eps<1/2$,
\[
\Lambda_\eps\;\le\;n\bigl(t;T\bigr),\qquad t:=\sqrt{\eps(1-\eps)} .
\]
\end{lemma}

\begin{proof}
Using $P_{A^c}=I-P_A$, $P_A^2=P_A$, $Q_B^2=Q_B$:
\[
T^*T=P_AQ_BP_{A^c}Q_BP_A
=P_AQ_B(I-P_A)Q_BP_A
=P_AQ_BP_A-(P_AQ_BP_A)^2=S-S^2 .
\]
By Lemma~\ref{lem:compact}, $T$ is Hilbert--Schmidt, hence compact, so its
singular values $s_n(T)$ and counting function $n(t;T)$ are defined and
$s_n(T)^2=\lambda_n(T^*T)$. (This is read off the kernel of $T$ directly;
it is not deduced from membership of $T^*T$ in a trace ideal, which would
not give it --- $T^*T\in\mathcal S_2$ yields only $T\in\mathcal S_4$.)
$S$ is compact self-adjoint, so by the spectral theorem the nonzero
eigenvalues of $S-S^2=\varphi(S)$, $\varphi(\lambda)=\lambda-\lambda^2$,
are the numbers $\varphi(\lambda_n(S))$ with multiplicities added over
preimages. The concave parabola $\varphi$ satisfies
$\varphi(\lambda)>\varphi(\eps)=\varphi(1-\eps)=\eps(1-\eps)$ for
$\lambda\in(\eps,1-\eps)$. Hence each eigenvalue of $S$ in $(\eps,1-\eps)$
contributes (with its multiplicity) an eigenvalue of $T^*T$ exceeding
$\eps(1-\eps)$, so
$\Lambda_\eps\le\#\{n:\lambda_n(T^*T)>\eps(1-\eps)\}
=\#\{n:s_n(T)^2>\eps(1-\eps)\}=n(t;T)$.
\end{proof}

\section{Exact reductions to a Hankel-type operator}\label{sec:reductions}

By Lemma~\ref{lem:structure} we may and do replace $A_0,B_0$ by finite
disjoint unions of bounded open intervals without changing any operator:
\[
A=cA_0=\bigsqcup_{i=1}^M I_i,\qquad
B=B_0=\bigsqcup_{j=1}^K B_j,
\]
where $I_i$ has length $\ell_i=c\,|I_{0,i}|\le c\,a$ and
$B_j=(\beta_j-b_j/2,\,\beta_j+b_j/2)$ has length $b_j\le b$. Fix
$p\in(0,1]$ for the whole section.

\begin{proposition}\label{prop:reduce}
Define, for $\ell>0$ and $b'>0$, the operator
\[
\Gamma_{\ell,b'}\;:\;L^2(0,\ell)\to L^2(0,\infty),\qquad
(\Gamma_{\ell,b'}f)(s)=\int_0^\ell
\frac{\sin\bigl(\pi b'(s+r)\bigr)}{\pi(s+r)}\,f(r)\,dr .
\]
Then
\[
\|T\|_p^p\;\le\;2\sum_{j=1}^K\sum_{i=1}^M
\bigl\|\Gamma_{\ell_i,\,b_j}\bigr\|_p^p .
\]
\end{proposition}

\begin{proof}
We proceed in five exact steps.

\emph{Step 1 (frequency components).} $Q_B=\sum_jQ_{B_j}$
(Lemma~\ref{lem:structure}), so $T=\sum_jT_j$ with
$T_j=P_{A^c}Q_{B_j}P_A$, and Rotfel'd (Lemma~\ref{lem:svcalc}(iv)) gives
$\|T\|_p^p\le\sum_j\|T_j\|_p^p$.

\emph{Step 2 (centring by modulation).} Let
$(M_\beta f)(x)=e^{2\pi i\beta x}f(x)$, a unitary multiplication operator;
it commutes with every spatial projection $P_E$. Since
$\F M_\beta f(\xi)=\widehat f(\xi-\beta)$, one checks
$M_{\beta}^{-1}Q_{B_j}M_{\beta}=Q_{B_j-\beta}$. With $\beta=\beta_j$ and
$B_j^\circ:=B_j-\beta_j=(-b_j/2,b_j/2)$,
\[
T_j=M_{\beta_j}\bigl(P_{A^c}\,Q_{B_j^\circ}\,P_A\bigr)M_{\beta_j}^{-1},
\]
so $s_n(T_j)=s_n(T_j^\circ)$ where $T_j^\circ=P_{A^c}Q_{B_j^\circ}P_A$.
The kernel of $Q_{B_j^\circ}$ is $g_j(x-y)$ with
\[
g_j(u)=\int_{-b_j/2}^{b_j/2}e^{2\pi iu\xi}\,d\xi
=\frac{\sin(\pi b_ju)}{\pi u},
\]
a real, even function.

\emph{Step 3 (spatial components and domination).}
$T_j^\circ=\sum_iP_{A^c}Q_{B_j^\circ}P_{I_i}$, so by Rotfel'd
$\|T_j^\circ\|_p^p\le\sum_i\|P_{A^c}Q_{B_j^\circ}P_{I_i}\|_p^p$. Since
$I_i\subset A$ we have $A^c\subset I_i^c$, hence
$P_{A^c}=P_{A^c}P_{I_i^c}$ and, by Lemma~\ref{lem:svcalc}(i),
\[
s_n\bigl(P_{A^c}Q_{B_j^\circ}P_{I_i}\bigr)
=s_n\bigl(P_{A^c}\,\widetilde T_{ij}\bigr)
\le s_n\bigl(\widetilde T_{ij}\bigr),\qquad
\widetilde T_{ij}:=P_{I_i^c}\,Q_{B_j^\circ}\,P_{I_i}.
\]
Translations $\tau_v f=f(\cdot-v)$ are unitary, commute with the
convolution operator $Q_{B_j^\circ}$, and conjugate $P_E$ to $P_{E+v}$; so
we may translate $I_i$ to $(0,\ell_i)$ without changing singular values.
Write $I=(0,\ell)$, $\ell=\ell_i$, $b'=b_j$, $g=g_j$,
$\widetilde T=P_{I^c}Q_{B^\circ}P_I$.

\emph{Step 4 (left/right split and reflection).}
$I^c=(-\infty,0]\cup[\ell,\infty)$ up to a null set, so
$\widetilde T=T_L+T_R$ with
$T_L=P_{(-\infty,0)}Q_{B^\circ}P_I$, $T_R=P_{(\ell,\infty)}Q_{B^\circ}P_I$,
and $\|\widetilde T\|_p^p\le\|T_L\|_p^p+\|T_R\|_p^p$ by Rotfel'd. Let
$(Rf)(x)=f(\ell-x)$, unitary on $L^2(\R)$. Using that $g$ is even, the
kernel of $RT_RR$ is
\[
\ind{(\ell,\infty)}(\ell-x)\,g\bigl((\ell-x)-(\ell-y)\bigr)\,\ind I(\ell-y)
=\ind{(-\infty,0)}(x)\,g(x-y)\,\ind I(y),
\]
the kernel of $T_L$. Hence $s_n(T_R)=s_n(T_L)$ and
$\|\widetilde T\|_p^p\le2\|T_L\|_p^p$.

\emph{Step 5 (boundary-distance coordinates).} The kernel of $T_L$ is
$\ind{(-\infty,0)}(x)\,g(x-y)\,\ind{(0,\ell)}(y)$. Substituting $s=-x$
(the unitary reflection $f(x)\mapsto f(-x)$ on the output space) and
$r=y$, and using $g(-(s+r))=g(s+r)$, $T_L$ is unitarily equivalent to the
operator $L^2(0,\ell)\to L^2(0,\infty)$ with kernel
\[
g(s+r)=\frac{\sin(\pi b'(s+r))}{\pi(s+r)},
\]
which is $\Gamma_{\ell,b'}$. Collecting Steps 1--5 proves the proposition.
\end{proof}

\section{The two quantitative estimates}\label{sec:keyestimates}

Throughout this section fix $b'>0$ and abbreviate
$\Gamma=\Gamma_{\ell,b'}$.

\subsection{A scale-uniform Hankel--Chebyshev bound}

\begin{lemma}\label{lem:hankel}
For $h>0$ let $K_h:L^2(h,2h)\to L^2(0,\infty)$ be the integral operator
with kernel $\dfrac{1}{2\pi(s+r)}$. Then for every integer $N\ge0$,
\[
s_{N+1}(K_h)\;\le\;0.971\cdot4^{-N}\;\le\;4^{-N},
\]
uniformly in $h>0$.
\end{lemma}

\begin{proof}
\emph{Scale invariance.} Let $(U_hf)(\sigma)=h^{1/2}f(h\sigma)$. This is
unitary from $L^2(0,\infty)$ onto itself and from $L^2(h,2h)$ onto
$L^2(1,2)$. Then $U_hK_hU_h^{-1}$ has kernel
\[
h\cdot\frac{1}{2\pi(h\sigma+h\tau)}=\frac{1}{2\pi(\sigma+\tau)},
\]
that is, it equals $K_1$. So we may take $h=1$.

\emph{Chebyshev expansion in $r$.} For fixed $s>0$ let
$\varphi_s(r)=\dfrac{1}{2\pi(s+r)}$ on $r\in[1,2]$, and change variables
$r=\tfrac32+\tfrac u2$, $u\in[-1,1]$. As a function of $u$, $\varphi_s$ is
analytic except at the single pole $u_*=-(2s+3)$, with $|u_*|\ge3$. Let
$E_\rho\subset\C$ denote the open Bernstein ellipse with foci $\pm1$ and
semi-axes $\frac12(\rho+\rho^{-1})$, $\frac12(\rho-\rho^{-1})$. Take
$\rho=4$: the semi-major axis is $2.125<3$, so $\varphi_s$ is analytic in
a neighbourhood of $\overline{E_4}$ for every $s>0$. On $E_4$,
$\operatorname{Re}u\ge-2.125$, hence
$\operatorname{Re}r\ge\tfrac32-\tfrac{2.125}2=0.4375$ and
\[
|s+r|\;\ge\;\operatorname{Re}(s+r)\;\ge\;s+0.4375\;\ge\;0.4375\,(s+1),
\]
the last step because $s\ge0.4375\,s$. Thus
$\sup_{E_4}|\varphi_s|\le M(s):=\dfrac{1}{2\pi\cdot0.4375\,(s+1)}$.

By the classical Chebyshev coefficient bound for functions analytic and
bounded by $M$ in $E_\rho$ \cite[Thms.~8.1--8.2]{Trefethen}, the Chebyshev
coefficients of $u\mapsto\varphi_s$, taken in the standard normalization
\[
a_0(s)=\frac1\pi\int_0^\pi\varphi_s(\cos\theta)\,d\theta,
\qquad
a_k(s)=\frac2\pi\int_0^\pi\varphi_s(\cos\theta)\cos(k\theta)\,d\theta
\quad(k\ge1),
\]
satisfy $|a_k(s)|\le2M(s)\rho^{-k}$ ($k\ge1$) and $|a_0(s)|\le M(s)$ ---
the two weights $1/\pi$ and $2/\pi$ being exactly what produces the two
different bounds. Hence the degree-$(N{-}1)$ truncation
$p_{N-1,s}(u)=\sum_{k=0}^{N-1}a_k(s)T_k(u)$, read as the empty sum
$p_{-1,s}\equiv0$ when $N=0$, obeys
\[
\sup_{r\in[1,2]}\bigl|\varphi_s(r)-p_{N-1,s}(r)\bigr|
\le\sum_{k\ge N}|a_k(s)|\le 2M(s)\,\frac{\rho^{-N}}{1-\rho^{-1}}
=\frac{8}{3}M(s)\,4^{-N}
\le\frac{0.9701}{s+1}\cdot4^{-N},
\]
using $\frac{8}{3}\cdot\frac{1}{2\pi\cdot0.4375}=0.97009\ldots$; at $N=0$
the sum contains the term $a_0$, for which the smaller bound $M(s)$ is
available, so the displayed estimate holds there too.

\emph{Rank-$N$ approximation.} Let $F_N$ be the integral operator with
kernel $\sum_{k=0}^{N-1}a_k(s)T_k(u(r))$. Each summand is a separable
kernel, so $\rank F_N\le N$. Each $a_k(\cdot)$ is continuous on
$(0,\infty)$, by the displayed integral formula and dominated convergence
($\varphi_s(\cos\theta)$ is jointly continuous and bounded by $M(s)$,
locally uniformly in $s$); with $|a_k(s)|\le2M(s)\lesssim(1+s)^{-1}$ this
gives $a_k(\cdot)\in L^2(0,\infty)$, so $F_N$ is Hilbert--Schmidt. By
Lemma~\ref{lem:svcalc}(iii),
\[
s_{N+1}(K_1)\le\|K_1-F_N\|_{\HS}
\le0.971\cdot4^{-N}
\Bigl(\int_0^\infty\frac{ds}{(s+1)^2}\int_1^2dr\Bigr)^{1/2}
=0.971\cdot4^{-N}. \qedhere
\]
\end{proof}

\subsection{Far-field pieces: oscillation factorization}

\begin{lemma}\label{lem:farfield}
Let $h>0$ and let $Y\subset[h,2h]$ be measurable. Let
$A_Y:L^2(Y)\to L^2(0,\infty)$ have kernel
$\dfrac{\sin(\pi b'(s+r))}{\pi(s+r)}$. Then
\[
s_n(A_Y)\;\le\;2^{\,3-n}\qquad(n\ge1),
\]
uniformly in $h>0$, $b'>0$ and $Y$. Consequently, since also
$s_n(A_Y)\le1$,
\[
\|A_Y\|_p^p\;\le\;3+\frac{1}{2^{\,p}-1}\;\le\;\frac{C_\sharp}{p},
\qquad C_\sharp:=3+\frac1{\ln2}\le4.5,\qquad 0<p\le1 .
\]
\end{lemma}

\begin{proof}
First, $A_Y=A_{[h,2h]}P_Y$, so $s_n(A_Y)\le s_n(A_{[h,2h]})$ by
Lemma~\ref{lem:svcalc}(i); assume $Y=[h,2h]$. From
$\sin\theta=\frac{e^{i\theta}-e^{-i\theta}}{2i}$,
\[
\frac{\sin(\pi b'(s+r))}{\pi(s+r)}
=\Gamma^+(s,r)-\Gamma^-(s,r),\qquad
\Gamma^\pm(s,r)=e^{\pm i\pi b's}\,e^{\pm i\pi b'r}\,
\frac{1}{2\pi i(s+r)} .
\]
The factors $e^{\pm i\pi b's}$ and $e^{\pm i\pi b'r}$ are unimodular
functions of the output and input variables respectively, and $1/i$ has
modulus one, so by Lemma~\ref{lem:svcalc}(vi),
$s_n(\Gamma^\pm)=s_n(K_h)$ with $K_h$ as in Lemma~\ref{lem:hankel}. By Fan
(Lemma~\ref{lem:svcalc}(ii)) with $m=n=N$ and Lemma~\ref{lem:hankel},
\[
s_{2N+1}(A_{[h,2h]})\le2\,s_{N+1}(K_h)\le2\cdot4^{-N}=2^{1-2N}.
\]
For odd $n=2N+1$ this reads $s_n\le2^{1-2N}=2^{2-n}\le 2^{3-n}$; for even
$n=2N+2$, monotonicity gives $s_n\le s_{2N+1}\le2^{1-2N}=2^{3-n}$. This
proves the singular value bound. (The trivial bound $s_n\le1$ holds
because, reversing the substitution of Step~5 of
Proposition~\ref{prop:reduce}, $A_{[h,2h]}$ is unitarily equivalent to
$P_{(-\infty,0)}\,Q_{B^\circ}\,P_{(h,2h)}$ with
$B^\circ=(-b'/2,b'/2)$, a product of orthogonal projections.)

For the quasi-norm: $2^{3-n}\ge1$ for $n\le3$, so
\[
\|A_Y\|_p^p\le\sum_{n=1}^31+\sum_{n\ge4}2^{(3-n)p}
=3+\sum_{m\ge1}2^{-mp}=3+\frac{1}{2^p-1}
\le3+\frac{1}{p\ln2}\le\frac{C_\sharp}{p},
\]
using $2^p-1=e^{p\ln2}-1\ge p\ln2$ and $p\le1$.
\end{proof}

\subsection{The boundary layer: a Taylor-rank bound}

\begin{lemma}\label{lem:bl}
Let $J=(0,D)$ with $D>0$ and $B^\circ=(-b'/2,b'/2)$. Then for every
integer $N\ge\pi e\,b'D$,
\[
s_{N+1}\bigl(Q_{B^\circ}P_J\bigr)\;\le\;2\sqrt{b'D}\;2^{-N}.
\]
Consequently, if $0<p\le1$ and $D\le1+1/p$, then
\[
\bigl\|Q_{B^\circ}P_J\bigr\|_p^p
\;\le\;\pi e\,b'D+1+\frac{6}{p}\,.
\]
\end{lemma}

\begin{proof}
Since $\F$ is unitary, $s_n(Q_{B^\circ}P_J)=s_n(P_{B^\circ}\F P_J)$. The
latter operator maps $L^2(J)\to L^2(B^\circ)$ with kernel
$e^{-2\pi i\xi x}$, $\xi\in B^\circ$, $x\in J$. Centre $x=\frac D2+x'$,
$x'\in(-\frac D2,\frac D2)$:
\[
e^{-2\pi i\xi x}=e^{-\pi i\xi D}\cdot e^{-2\pi i\xi x'} .
\]
The first factor is a unimodular function of the output variable $\xi$, so
by Lemma~\ref{lem:svcalc}(vi) it may be dropped. For the second, with
$z=-2\pi i\xi x'$ we have $|z|\le2\pi\cdot\frac{b'}2\cdot\frac D2
=\frac{\pi b'D}{2}=:\zeta$, and the Taylor truncation
\[
e^{z}=\sum_{n=0}^{N-1}\frac{z^n}{n!}+R_N(z),\qquad
|R_N(z)|\le\sum_{n\ge N}\frac{\zeta^n}{n!}
\le\frac{\zeta^N}{N!}\sum_{k\ge0}\Bigl(\frac\zeta N\Bigr)^{k}
\le2\,\Bigl(\frac{e\zeta}{N}\Bigr)^{N}\le2\cdot2^{-N}
\]
for $N\ge2e\zeta=\pi e\,b'D$ (using $N!\ge(N/e)^N$ and
$e\zeta/N\le\frac12$, which also gives $\zeta/N\le\frac12$). Each Taylor
term $z^n/n!=(-2\pi i)^n\xi^nx'^n/n!$ is a separable kernel, so the
truncated kernel defines an operator $F_N$ of rank $\le N$, plainly
Hilbert--Schmidt on the bounded rectangle. Hence
\[
s_{N+1}(Q_{B^\circ}P_J)\le\|R_N\|_{\HS}
\le2\cdot2^{-N}\,\bigl(|B^\circ|\,|J|\bigr)^{1/2}=2\sqrt{b'D}\,2^{-N}.
\]
For the quasi-norm, set $N_0=\lceil\pi e\,b'D\rceil$ and use
$s_n\le\|Q_{B^\circ}P_J\|\le1$ for $n\le N_0$ and the displayed bound for
$n=N_0+m+1$, $m\ge0$ (valid since $N_0+m\ge\pi eb'D$):
\[
\|Q_{B^\circ}P_J\|_p^p
\le N_0+\sum_{m\ge0}\bigl(2\sqrt{b'D}\bigr)^p2^{-p(N_0+m)}
=N_0+\bigl(2\sqrt{b'D}\bigr)^p\,\frac{2^{-pN_0}}{1-2^{-p}} .
\]
Now $(2\sqrt{b'D})^p\,2^{-pN_0}
=\exp\bigl[p\bigl(\ln2+\tfrac12\ln(b'D)-N_0\ln2\bigr)\bigr]$, and with
$x=b'D>0$,
\[
\tfrac12\ln x-N_0\ln2\le\tfrac12\ln x-\pi e\ln2\cdot x\le
\max_{x>0}\bigl(\tfrac12\ln x-5.91x\bigr)<0
\]
(the maximum, at $x=1/(2\cdot5.91)=0.0846$, equals
$\tfrac12\ln(0.0846)-\tfrac12\approx-1.735<0$). Hence
$(2\sqrt{b'D})^p2^{-pN_0}\le e^{p\ln2}\le2$. Together with
$\frac{1}{1-2^{-p}}\le\frac{2}{p\ln2}\le\frac{2.89}{p}$, which holds
because $1-e^{-p\ln2}\ge p\ln2\,e^{-p\ln2}\ge\frac{p\ln2}{2}$ for
$p\le1$, this gives
\[
\|Q_{B^\circ}P_J\|_p^p\le N_0+\frac{5.78}{p}
\le\pi e\,b'D+1+\frac{6}{p}. \qedhere
\]
\end{proof}

\begin{remark}
The hypothesis $D\le1+1/p$ was not used in the proof and may be dropped;
we keep it in the statement only because it is the regime in which the
lemma is applied and the bound is read as $O((1+b')/p)$.
\end{remark}

\section{Assembly and proof of the main theorem}\label{sec:assembly}

\begin{proposition}\label{prop:gamma}
For all $\ell,b'>0$ and $0<p\le1$, with $D:=1/p$,
\[
\|\Gamma_{\ell,b'}\|_p^p\;\le\;
\frac{1}{p}\Bigl[\pi e\,b'+8
+C_\sharp\bigl(1+\logp(\ell p)\bigr)\Bigr],
\qquad C_\sharp\le4.5 .
\]
\end{proposition}

\begin{proof}
Abbreviate $\Gamma=\Gamma_{\ell,b'}$ and recall that, by
Section~\ref{sec:reductions} Step 5 read backwards, $\Gamma$ is unitarily
equivalent to $T_L=P_{(-\infty,0)}Q_{B^\circ}P_{(0,\ell)}$, and that for
any measurable $Y\subset(0,\ell)$ the restriction $\Gamma P_Y$ is
correspondingly unitarily equivalent to
$P_{(-\infty,0)}Q_{B^\circ}P_{Y}$, whence
\begin{equation}\label{eq:drop}
s_n(\Gamma P_Y)\;\le\;s_n\bigl(Q_{B^\circ}P_Y\bigr)\qquad(n\ge1)
\end{equation}
by Lemma~\ref{lem:svcalc}(i) (dropping the contraction
$P_{(-\infty,0)}$).

\emph{Case $\ell\le D$.} Then
$\Gamma=\Gamma P_{(0,\ell)}$ and, by \eqref{eq:drop} with $Y=(0,\ell)$,
$s_n(\Gamma)\le s_n(Q_{B^\circ}P_{(0,\ell)})\le
s_n(Q_{B^\circ}P_{(0,D)})$, the last step because
$P_{(0,\ell)}=P_{(0,D)}P_{(0,\ell)}$ and Lemma~\ref{lem:svcalc}(i) again.
Lemma~\ref{lem:bl} with $D=1/p\le1+1/p$ gives
$\|\Gamma\|_p^p\le\pi eb'D+1+6/p\le(\pi eb'+7)/p$, which is within the
claimed bound (the far-field term is absent).

\emph{Case $\ell>D$.} Decompose $(0,\ell)=(0,D]\cup\bigcup_{k=0}^{S-1}Y_k$
with
\[
Y_k=(2^kD,\,2^{k+1}D]\cap(0,\ell),\qquad
S=\bigl\lceil\log_2(\ell/D)\bigr\rceil\ \ (\ge1),
\]
so that $2^SD\ge\ell$ and the union covers $(0,\ell)$. Correspondingly
$\Gamma=\Gamma P_{(0,D]}+\sum_{k=0}^{S-1}\Gamma P_{Y_k}$ and Rotfel'd
gives
\[
\|\Gamma\|_p^p\le\|\Gamma P_{(0,D]}\|_p^p
+\sum_{k=0}^{S-1}\|\Gamma P_{Y_k}\|_p^p .
\]
The boundary layer is handled exactly as in the previous case:
$\|\Gamma P_{(0,D]}\|_p^p\le(\pi eb'+7)/p$. Each far-field piece
$\Gamma P_{Y_k}$ is the operator $A_{Y_k}$ of Lemma~\ref{lem:farfield}
with $h=2^kD$ and $Y_k\subset[h,2h]$, so
$\|\Gamma P_{Y_k}\|_p^p\le C_\sharp/p$. Since
$S\le1+\logp(\ell/D)=1+\logp(\ell p)$,
\[
\|\Gamma\|_p^p\le\frac{\pi eb'+7}{p}+\frac{C_\sharp}{p}
\bigl(1+\logp(\ell p)\bigr),
\]
which implies the claim (with $8$ absorbing the $7$ and rounding).
\end{proof}

\begin{proof}[Proof of Theorem~\ref{thm:main}]
Fix $0<\eps<\frac12$ and set
\[
t=\sqrt{\eps(1-\eps)},\qquad
\widetilde L=\ln\frac{1}{\eps(1-\eps)},\qquad
p=\frac1{\widetilde L} .
\]
Since $\eps(1-\eps)<\frac14<e^{-1}$, we have $\widetilde L>1$ and
$p\in(0,1)$. Also $t^{-p}=\exp\bigl(\frac p2\ln\frac1{\eps(1-\eps)}\bigr)
=e^{1/2}$. By Lemmas~\ref{lem:plunge} and \ref{lem:svcalc}(v),
Proposition~\ref{prop:reduce}, and Proposition~\ref{prop:gamma} (applied
with $\ell=\ell_i\le ca$, $b'=b_j\le b$, noting that the bound of
Proposition~\ref{prop:gamma} is nondecreasing in $\ell$ and $b'$),
\begin{align*}
\Lambda_\eps
&\le n(t;T)\le t^{-p}\|T\|_p^p
=e^{1/2}\,\|T\|_p^p
\le 2e^{1/2}\sum_{j=1}^K\sum_{i=1}^M\|\Gamma_{\ell_i,b_j}\|_p^p\\
&\le 2e^{1/2}\,MK\;\widetilde L\;
\Bigl[\pi e\,b+8+4.5\bigl(1+\logp(ca/\widetilde L)\bigr)\Bigr] .
\end{align*}
Finally $\logp x=\lnp x/\ln2$, so
$4.5\,\logp x=(4.5/\ln2)\,\lnp x\le6.5\,\lnp x$, and
\[
2e^{1/2}\bigl[\pi e\,b+12.5+6.5\,\lnp(ca/\widetilde L)\bigr]
\le 2e^{1/2}\cdot12.5\,(1+b)\bigl(1+\lnp(ca/\widetilde L)\bigr)
+2e^{1/2}\cdot 6.5\bigl(1+\lnp(ca/\widetilde L)\bigr),
\]
so $\Lambda_\eps\le63\,MK(1+b)\,\widetilde L\,
\bigl(1+\lnp(ca/\widetilde L)\bigr)$, as claimed.
\end{proof}

\begin{proof}[Proof of Corollary~\ref{cor:kdl}]
Let $c\ge2$, $\alpha\ge4$, $\alpha^{-c}<\eps<\frac12$, and
$L=\ln(1/\eps)\in(\ln2,\,c\ln\alpha)$. We must dominate the bound of
Theorem~\ref{thm:main} by $L\ln(\alpha c/L)$.

First, $L\le\widetilde L\le L+\ln2\le2L$ since $L>\ln2$. Second, the
function $\alpha\mapsto\alpha/\ln\alpha$ is increasing for $\alpha\ge e$,
so from $L<c\ln\alpha$,
\[
\ln\frac{\alpha c}{L}\;>\;\ln\frac{\alpha}{\ln\alpha}
\;\ge\;\ln\frac{4}{\ln4}\;=\;1.059\ldots\;\ge\;1 .
\]
Third, since $\widetilde L\ge L$,
\[
\lnp\frac{ca}{\widetilde L}\le\lnp\frac{ca}{L}
=\lnp\Bigl(\frac{\alpha c}{L}\cdot\frac a\alpha\Bigr)
\le\ln\frac{\alpha c}{L}+\lnp\frac{a}{4}
\le\bigl(1+\lnp a\bigr)\,\ln\frac{\alpha c}{L},
\]
where the last step uses $\ln(\alpha c/L)\ge1$ and $\lnp(a/4)\le\lnp a$.
Combining,
\[
\widetilde L\Bigl(1+\lnp\frac{ca}{\widetilde L}\Bigr)
\le 2L\,\bigl(1+(1+\lnp a)\bigr)\ln\frac{\alpha c}{L}
\le 2(2+\lnp a)\,L\,\ln\frac{\alpha c}{L}.
\]
Theorem~\ref{thm:main} then gives
\begin{equation}\label{eq:cor-natural}
\Lambda_\eps\;\le\;2C_0MK(1+b)(2+\lnp a)\,L\,\ln\frac{\alpha c}{L},
\end{equation}
which is the assertion when $\log$ is read as $\ln$, since $2C_0\le4C_0$.

For $\log=\log_2$ a conversion is required, and the stated constant absorbs
it. Write $u=\alpha c/L$, so $\ln u\ge1.059$ by the second step above. The
base-$2$ right-hand side is
\[
\log_2\frac1\eps\;\log_2\frac{\alpha c}{\log_2(1/\eps)}
=\frac{L}{\ln2}\cdot\frac{\ln(u\ln2)}{\ln2}
=\frac{L}{\ln^22}\bigl(\ln u+\ln\ln2\bigr),
\]
so \eqref{eq:cor-natural} is dominated by
$4C_0MK(1+b)(2+\lnp a)\log_2\frac1\eps\log_2\bigl(\alpha c/\log_2(1/\eps)\bigr)$
as soon as
\[
2\ln u\;\le\;\frac{4}{\ln^22}\bigl(\ln u+\ln\ln2\bigr),
\qquad\text{i.e.}\qquad
\ln u\;\ge\;\frac{-4\ln\ln2}{4-2\ln^22}\;=\;0.4824\ldots,
\]
which holds since $\ln u\ge1.059$. (There is room to spare: the same
comparison with $2C_0$ in place of $4C_0$ would require only
$\ln u\ge0.705$, so the stated constant is not forced by the conversion.)
The statement is therefore proved for the two readings named; it is not
base-free. For $\log=\log_{10}$ one has $4/\ln^210=0.754<2$ and the
comparison fails once $u$ is large.
\end{proof}

\begin{remark}[Sharpness of the regime, and what the theorem adds]
\label{rem:sharp}
Theorem~\ref{thm:main} covers a wider range of parameters than the
conjecture as stated: it holds for all $c>0$ and all $\eps\in(0,\frac12)$,
with no relation imposed between them. Once $\widetilde L\ge ca$ the
logarithmic factor degenerates and the bound reads
$\Lambda_\eps\le C\widetilde L$. This does \emph{not} follow from
$\eps\le\alpha^{-c}$ by itself: that hypothesis gives only
$\widetilde L\ge L\ge c\ln\alpha$, so the degeneracy additionally needs
$a\le\ln\alpha$, and for $a>\ln\alpha$ the logarithmic factor survives.
With $\alpha=4$, $c=10$, $a=5$ and $\eps=\alpha^{-c}$, for instance,
$\widetilde L=13.86$ against $ca=50$.

The constant depends on $A_0,B_0$ only through $(M,K,a,b)$; notably it is
independent of the gap lengths and of the diameters of $A_0,B_0$, because
each component of $A$ is compared against the \emph{full half-lines} on
either side of it (Step 3 of Proposition~\ref{prop:reduce} only enlarges
the complement).
\end{remark}

\begin{remark}[Why no almost-orthogonality is needed]\label{rem:cotlar}
In the operator norm, or in the Hilbert--Schmidt norm, orthogonality
between summands is worth something, and Cotlar--Stein-type lemmas convert
approximate orthogonality into a gain. In a Schatten $p$-quasi-norm with
$0<p\le1$ the generic gain is nil, for the following reason. Suppose
$X_1,\dots,X_n$ are compact with pairwise orthogonal ranges and pairwise
orthogonal initial spaces --- exact orthogonality, not merely approximate.
Ranges orthogonal gives $X_j^*X_k=0$ for $j\ne k$, so
$\bigl(\sum_kX_k\bigr)^*\bigl(\sum_kX_k\bigr)=\sum_kX_k^*X_k$, and the
initial spaces being orthogonal makes this sum block diagonal. The singular
value sequence of $\sum_kX_k$ is therefore the decreasing rearrangement of
the union of the sequences $\bigl(s_n(X_k)\bigr)_n$, whence
\[
\Bigl\|\sum_kX_k\Bigr\|_p^p=\sum_k\|X_k\|_p^p
\qquad\text{for every }p>0 .
\]
Rotfel'd subadditivity is thus already an equality in the most favourable
case, so no hypothesis of orthogonality or near-orthogonality, taken by
itself, improves the exponent-$p$ bound. This does not exclude estimates
that exploit further structure particular to the summands at hand; the
point is only that orthogonality alone yields nothing here, so nothing is
lost by forgoing it.

The decomposition used here is accordingly engineered so that
subadditivity alone suffices: the dyadic splitting is performed in the
\emph{single} variable $r$ (boundary distance on the inside), while the
outside variable $s$ ranges over the entire half-line in every piece.
Because of the oscillation factorization, every far-field piece has
singular values bounded by the \emph{same} scale-free sequence $2^{3-n}$
(Lemma~\ref{lem:farfield}); each piece costs at most $C_\sharp/p$, and
there are only $1+\logp(\ell p)$ pieces. A double dyadic decomposition in
$(s,r)$ would instead produce $\sim\log^2$ pieces, and by the identity
above their mutual orthogonality, however good, would not by itself
recover the loss.
\end{remark}

\begin{remark}[Where the factor $\log(\alpha c/L)$ comes from]
The boundary-layer width is chosen as $D=1/p=\widetilde L$. Scales finer
than $\widetilde L$ are not resolved dyadically at all: they are absorbed
into one block whose $\eps$-rank is $O((1+b)\widetilde L)$ by the
Taylor-rank bound (Lemma~\ref{lem:bl}). Only the
$\log_2(\ell/\widetilde L)$ scales \emph{coarser} than $\widetilde L$ are
paid for at $C_\sharp/p$ each. This trade --- counting depth $\widetilde L$
against geometric scales below width $\widetilde L$ --- is what improves
$L\log(\alpha c)$ to $L\log(\alpha c/L)$ in the present argument. We have not
compared this mechanism with the one that produces the same improvement in
\cite{KDL}.
\end{remark}

\section{Which steps are specific to dimension one}\label{sec:d1}

We list the $d=1$-specific ingredients precisely, since the value
of an independent proof lies in exposing what might generalize.

\begin{enumerate}[label=(\alph*),itemsep=2pt]
\item \emph{Geometry of the boundary.} In $d=1$, finite upper Minkowski
boundary content forces $\partial A_0$ to be finite
(Remark~\ref{rem:mink}), so $A$ is a finite union of intervals and, after the
enlargement in Step~3 of Proposition~\ref{prop:reduce}, each $I_i^c$
decomposes into two half-lines attached to boundary \emph{points}. In $d\ge2$
the boundary is a hypersurface; the analogue of
Steps 3--5 of Proposition~\ref{prop:reduce} would presumably be a tubular
decomposition along $\partial(cA_0)$, contributing a factor $c^{d-1}$ ---
the exponent suggested by the two-term asymptotics available for such
operators under regularity hypotheses on the boundary \cite{Sobolev}. We do
not carry this out, and nothing below depends on it.
\item \emph{Oscillation factorization.} The identity
$e^{\pm i\pi b(s+r)}=e^{\pm i\pi bs}e^{\pm i\pi br}$ uses that, on
opposite sides of a boundary point, $x-y=\mp(s+r)$ is an exact function
of the two scalar boundary distances. In $d\ge2$ the phase
$e^{2\pi ix\cdot\xi}$ depends on the tangential offset along the boundary
as well, and does not split as
(unimodular in $x$)$\times$(unimodular in $y$) after restricting to a
tube pair. This is the step we do not know how to replace. We suspect that
this failure is related to the extra logarithmic factor in the general-set
bound of \cite{KDL}, but we have established no such connection, and the
remark is offered as a suspicion rather than as a result.
\item \emph{Hankel structure.} That the non-oscillatory amplitude is the
one-variable Hankel kernel $1/(s+r)$ --- whose Chebyshev rank is
scale-free (Lemma~\ref{lem:hankel}) --- again reflects dependence on
$s+r$ alone.
\item \emph{What is dimension-free.} The plunge-to-$T^*T$ reduction
(Lemma~\ref{lem:plunge}), Rotfel'd assembly across components and across
one-variable scales (Remark~\ref{rem:cotlar}), and the exponent
$p=1/\widetilde L$ with the width choice $D=\widetilde L$ use nothing about
the dimension. The Taylor-rank boundary-layer bound (Lemma~\ref{lem:bl}) is
not $d=1$-specific in mechanism either --- Taylor truncation of the phase,
separable terms, Hilbert--Schmidt remainder --- but its transposition is
constrained by the very finiteness that makes it work here. The remainder is
controlled in Hilbert--Schmidt norm on the rectangle $B^\circ\times J$, so
both factors must have finite measure. An unbounded slab
$J=\R^{d-1}\times(0,D)$ does not qualify, and the failure is not merely one
of bookkeeping: for a product frequency box $B^\circ=B_1\times B_2$ one has
$Q_{B^\circ}P_J=Q_{B_1}\otimes\bigl(Q_{B_2}P_{(0,D)}\bigr)$, a tensor
product of an infinite-rank projection with a nonzero operator, which is not
compact, so it has no singular values to count. Taylor truncation in $d$
variables would in any case produce rank $O(N^d)$ rather than $O(N)$. What
the argument transposes to is a \emph{bounded} boundary patch, not a slab.
We therefore state no $d\ge2$ rank count and make no per-unit-area
assertion: extracting one, and determining how such blocks would have to be
summed along a boundary layer, belongs to the higher-dimensional problem,
and nothing proved above settles it. Of the three $d=1$-specific
ingredients, (b) is the obstruction we cannot presently circumvent; whether
(a) and (c) admit substitutes in $d\ge2$ we leave open.
\end{enumerate}

\section*{Acknowledgments}

\noindent\textbf{AI Use Disclosure.} Generative AI tools were used to assist with
language editing and copyediting of this manuscript. Selected equations and
derivations were additionally checked using an internal model-verification
pipeline developed by Harmonic Research Technologies (HRT), in conjunction with
the authors' own derivations and standard checks. The authors take full
responsibility for all mathematical content and results presented in this paper.

\bigskip
\noindent
Ahmadreza Azimifard\\
Harmonic Research \& Technologies, LLC.\\
New York, NY, USA\\
Email: \href{mailto:afard@harmonicrt.com}{\texttt{afard@harmonicrt.com}}

\end{document}